\documentclass[12pt]{amsart}

\usepackage{amsmath,amssymb,amsbsy,amsfonts,amsthm,latexsym,
                        amsopn,amstext,amsxtra,euscript,amscd,mathrsfs,color,bm,cite}

\usepackage[english]{babel}
\usepackage{todonotes} 
\usepackage{url}

\usepackage[margin=3.5cm]{geometry}

\usepackage[colorlinks,linkcolor=blue,anchorcolor=blue,citecolor=blue,backref=page]{hyperref}

\newtheorem{theorem}{Theorem}
\newtheorem{lemma}[theorem]{Lemma}

\newtheorem{cor}[theorem]{Corollary}

\theoremstyle{definition}
\newtheorem*{ack}{Acknowledgements}

\numberwithin{equation}{section}
\numberwithin{theorem}{section}

\numberwithin{figure}{section}

\newfont{\teneufm}{eufm10}
\newfont{\seveneufm}{eufm7}
\newfont{\fiveeufm}{eufm5}
\newfam\eufmfam
                \textfont\eufmfam=\teneufm \scriptfont\eufmfam=\seveneufm
                \scriptscriptfont\eufmfam=\fiveeufm
\def\bbbc{{\mathchoice {\setbox0=\hbox{$\displaystyle\rm C$}\hbox{\hbox
to0pt{\kern0.4\wd0\vrule height0.9\ht0\hss}\box0}}
{\setbox0=\hbox{$\textstyle\rm C$}\hbox{\hbox
to0pt{\kern0.4\wd0\vrule height0.9\ht0\hss}\box0}}
{\setbox0=\hbox{$\scriptstyle\rm C$}\hbox{\hbox
to0pt{\kern0.4\wd0\vrule height0.9\ht0\hss}\box0}}
{\setbox0=\hbox{$\scriptscriptstyle\rm C$}\hbox{\hbox
to0pt{\kern0.4\wd0\vrule height0.9\ht0\hss}\box0}}}}
\def\bbbq{{\mathchoice {\setbox0=\hbox{$\displaystyle\rm
Q$}\hbox{\raise 0.15\ht0\hbox to0pt{\kern0.4\wd0\vrule
height0.8\ht0\hss}\box0}} {\setbox0=\hbox{$\textstyle\rm
Q$}\hbox{\raise 0.15\ht0\hbox to0pt{\kern0.4\wd0\vrule
height0.8\ht0\hss}\box0}} {\setbox0=\hbox{$\scriptstyle\rm
Q$}\hbox{\raise 0.15\ht0\hbox to0pt{\kern0.4\wd0\vrule
height0.7\ht0\hss}\box0}} {\setbox0=\hbox{$\scriptscriptstyle\rm
Q$}\hbox{\raise 0.15\ht0\hbox to0pt{\kern0.4\wd0\vrule
height0.7\ht0\hss}\box0}}}}
\def\bbbt{{\mathchoice {\setbox0=\hbox{$\displaystyle\rm
T$}\hbox{\hbox to0pt{\kern0.3\wd0\vrule height0.9\ht0\hss}\box0}}
{\setbox0=\hbox{$\textstyle\rm T$}\hbox{\hbox
to0pt{\kern0.3\wd0\vrule height0.9\ht0\hss}\box0}}
{\setbox0=\hbox{$\scriptstyle\rm T$}\hbox{\hbox
to0pt{\kern0.3\wd0\vrule height0.9\ht0\hss}\box0}}
{\setbox0=\hbox{$\scriptscriptstyle\rm T$}\hbox{\hbox
to0pt{\kern0.3\wd0\vrule height0.9\ht0\hss}\box0}}}}
\def\bbbs{{\mathchoice
{\setbox0=\hbox{$\displaystyle     \rm S$}\hbox{\raise0.5\ht0\hbox
to0pt{\kern0.35\wd0\vrule height0.45\ht0\hss}\hbox
to0pt{\kern0.55\wd0\vrule height0.5\ht0\hss}\box0}}
{\setbox0=\hbox{$\textstyle        \rm S$}\hbox{\raise0.5\ht0\hbox
to0pt{\kern0.35\wd0\vrule height0.45\ht0\hss}\hbox
to0pt{\kern0.55\wd0\vrule height0.5\ht0\hss}\box0}}
{\setbox0=\hbox{$\scriptstyle      \rm S$}\hbox{\raise0.5\ht0\hbox
to0pt{\kern0.35\wd0\vrule height0.45\ht0\hss}\raise0.05\ht0\hbox
to0pt{\kern0.5\wd0\vrule height0.45\ht0\hss}\box0}}
{\setbox0=\hbox{$\scriptscriptstyle\rm S$}\hbox{\raise0.5\ht0\hbox
to0pt{\kern0.4\wd0\vrule height0.45\ht0\hss}\raise0.05\ht0\hbox
to0pt{\kern0.55\wd0\vrule height0.45\ht0\hss}\box0}}}}
\def\bbbz{{\mathchoice {\hbox{$\sf\textstyle Z\kern-0.4em Z$}}
{\hbox{$\sf\textstyle Z\kern-0.4em Z$}} {\hbox{$\sf\scriptstyle
Z\kern-0.3em Z$}} {\hbox{$\sf\scriptscriptstyle Z\kern-0.2em
Z$}}}}

\def\squareforqed{\hbox{\rlap{$\sqcap$}$\sqcup$}}
\def\qed{\ifmmode\squareforqed\else{\unskip\nobreak\hfil
\penalty50\hskip1em\null\nobreak\hfil\squareforqed
\parfillskip=0pt\finalhyphendemerits=0\endgraf}\fi}

\def\cB{{\mathcal B}}

\def\cI{{\mathcal I}}

\def\cS{{\mathcal S}}

\def\le{\leqslant}
\def\leq{\leqslant}
\def\ge{\geqslant}
\def\leq{\leqslant}

\def\va{{\mathbf{a}}}
\def\vb{{\mathbf{b}}}
\def\vc{{\mathbf{c}}}
\def\vn{{\mathbf{n}}}

\def \fB{\mathfrak{B}}

\def \fW{\mathfrak{W}}

\newcommand\ve\varepsilon
\newcommand{\ZZ}{\mathbb{Z}}
\renewcommand{\geq}{\geqslant}
\renewcommand{\leq}{\leqslant}

\newcommand{\ignore}[1]{}

\def\e{\mathbf{e}}

\def \Z{\mathbb{Z}}

\def \R{\mathbb{R}}

\def \N{\mathbb{N}}
\def \Z{\mathbb{Z}}

\def\mand{\qquad\mbox{and}\qquad}

\def\\{\cr}
\def\({\left(}
\def\){\right)}

\def\e{\mathbf{e}}

\def\cFH{\mathcal{F}_{k}(H)}
\def\cFHs{\mathcal{F}^*_{k}(H)}
\def\cGH{\mathcal{G}_{g}(H)}

\def\UkH{U_{k}(m,H,N)}
\def\WgH{W_{g}(m,H,N)}
\def\BH{\cB_{k}(H)}
\def\IH{\cI_g(H)}

\begin{document}

\title{Square-free values of random polynomials} 

\author{Tim D. Browning}
\address{IST Austria, Am Campus 1, 3400 Klosterneuburg, Austria}
\email{tdb@ist.ac.at}

\author[I. E. Shparlinski] {Igor E. Shparlinski}
\address{School of Mathematics and Statistics, University of New South Wales, Sydney, NSW 2052, Australia}
\email{igor.shparlinski@unsw.edu.au}

\begin{abstract}  
The question of whether or not a given integral polynomial takes infinitely many square-free values has only been addressed unconditionally for polynomials of degree at most $3$. We address this question, on average, for polynomials of arbitrary  degree. 
\end{abstract}


\date{\today}
\maketitle

\thispagestyle{empty}
\setcounter{tocdepth}{1}
\tableofcontents

\section{Introduction} 

Let 
$f \in \Z[X]$ 
be an  irreducible  polynomial of degree $k$, without a fixed square divisor.   We denote by 
$S_f(N)$ the number of positive integers $n \le N$ such 
that $f(n)$ is square-free.  It is  expected that 
\begin{equation}
\label{eq:Conj}
S_f(N) = c_f N + o(N),
\end{equation}
as $N \to \infty$, where
\[
c_f = \prod_{p~\mathrm{prime}}\(1 - \frac{\rho_f(p^2)}{p^2}\)
\]  
and $\rho_f(m)=\#\{n\in \Z/m\Z: f(n) \equiv 0 \bmod m\}$, for any positive integer $m$. When $k\leq 3$ this expectation follows from pioneering work of Hooley~\cite{hooley1}.  
(In fact, when $k=3$, Reuss~\cite{reuss} has produced an asymptotic formula for $S_f(N)$ with a power saving error term.)
However for polynomials of degree $k \geq 4$ we only 
have a conditional treatment under the $abc$-conjecture, thanks to work of Granville 
\cite{Gra}.

In this paper, for $k\geq 4$, we lend support to the expectation~\eqref{eq:Conj} by showing that it holds for almost all polynomials of degree $k$, when they are ordered by naive height.   This sits in the framework of a great deal of recent work aimed at understanding the average size  of various arithmetic functions over the values of random polynomials, with a focus on the von Mangoldt function 
$\Lambda$, the Liouville function $\lambda$ and variants of the $r$-function~\cite{BST, BaZh,nick, FoZh,SkSo}.

For positive integers   $H$ and   $k\geq 2$, let 
\[
\cFH = \{a_0+\cdots+a_kX^k\in \Z[X]:~ 
(a_0, a_1, \ldots, a_k) \in \cB_{k}(H)\}, 
\]  
where 
\begin{align*}
\BH = \left\{(a_0,\ldots, a_k) \in \Z^{k+1}:
\begin{array}{l}
\gcd(a_0, a_1, \ldots, a_k)  = 1  \\
 |a_0 |,|a_1 |,\ldots , |a_k| \le H
 \end{array}
\right\}.
\end{align*}
For the problem of counting square-free values of polynomials, 
we 
are primarily be interested in the largest allowable  range of $H$, with respect to $N$, for which we can prove the existence of $\delta>0$ such that 
\begin{equation}\label{eq:target}
\frac{1}{\#\cFH}\sum_{f \in \cFH} \left| S_f(N) - c_f N\right|
\ll N^{1-\delta}.
\end{equation}
(See Section~\ref{sec:not} for a precise definition of the symbol $\ll$.) 
This
confirms that~\eqref{eq:Conj} holds unconditionally on 
average over the polynomials of naive height at most $H$. 
Throughout our work we  
allow all implied constants, including the implied constant in~\eqref{eq:target} to depend on the degree $k$, with any further dependence indicated by an appropriate subscript. 
Obviously, we would like to be able to take $H$ as small as possible, with respect to $N$, in assessing the validity of~\eqref{eq:target}.

Let $A\geq 1$ and let $\ve>0$. 
The second author has shown in~\cite[Theorem~1.1]{Shp} that
there exists $\delta >0$, depending only on $\ve$ such that~\eqref{eq:target} 
holds provided that 
\[
 N^A \ge H\geq N^{k-1+\varepsilon},
 \]   
and provided that we allow the implied constant to depend on $A$ and $\ve.$
In this paper we revisit this argument using tools 
from the geometry of numbers and the determinant method, in order to increase  the range of $H$, as follows. 
 
\begin{theorem}
\label{thm:Av F'} Let $A\geq 1$ and $k\ge 4$ 
be fixed.  Assume that 
\[
 N^A \ge H\geq N^{k-3+\varepsilon},
 \]   
 for $\ve>0$. Then there exists $\delta >0$, depending only on $\ve$, such that~\eqref{eq:target} holds, 
 with the implied constant depending only on $A$, $k$ and $\ve$.
\end{theorem}

This is our main result. In fact, 
in Theorem~\ref{thm:Av F} 
we 
also  prove an upper bound for the left hand side of~\eqref{eq:target}, from which 
Theorem~\ref{thm:Av F'} follows. 
An alternative approach to this problem is available through modifying the proof of~\cite[Theorem~2.2]{BST} to treat the function
\[
F(n)=\mu^2(n)-\sum_{\substack{d^2\mid n\\ d\leq D}}
\mu(d),
\]   
for suitable $D\geq 1$. This is carried out in forthcoming work of Jelinek 
\cite{jelinek}.
 While this approach  does not seem to offer an improvement over the range of  $H$ in 
Theorem~\ref{thm:Av F'}, it does allow one to average over only two of the coefficients.

Inspired by recent work on Vinogradov's mean value theorem we can also treat a related problem in which we only vary one coefficient. Let 
\[
g(X) = b_1X+\cdots+b_kX^k \in \Z[X]
\]  
be given and note that 
 $g(0)=0$. Then we may consider the set  
\[
\cGH = \{a +g(X)\in \Z[X]:~  a \in \IH\},
\]  
where 
\[
\IH= \{a \in \Z:~\gcd(a, b_1, \ldots, b_k)  = 1,  \  |a| \le H\}. 
\]  
While the approach of~\cite{Shp} also applies to polynomials from $\cGH$, 
 we 
supplement it with some new bounds for  residues of polynomials falling in short intervals, 
which complement those of~\cite{BMS, CCGHSZ, CGOS,G-MRST, KerrMoh}.  
To formulate the result, we define
\begin{equation}
\label{eq:eta-k}
\eta(k) = \begin{cases} 2^{-k+1}, & \text{if}\ 2 \le k \le 5,\\
1/(k(k-1)), & \text{if}\ k\ge 6,
\end{cases}
\end{equation}
with which notation we have following result. 

\begin{theorem}
\label{thm:Av G'} 
Let  $k\ge 2$ 
be fixed and assume that 
\[
H \ge  
N^{(k-1)/2+\eta(k)+\ve},
\] 
 for $\ve>0$.
Then, for a fixed polynomial $g \in \Z[X]$ of degree $k$, there exists 
 $\delta >0$  depending only on $\ve$
such that 
\[
\frac{1}{\#\cGH}\sum_{f \in \cGH} \left| S_f(N) - c_f N\right| \ll  N^{1-\delta}, 
\]
with the implied constant depending only on $A$, $k$ and $\ve$.
\end{theorem}

Again, in Theorem~\ref{thm:Av G} we 
prove a version with an explicit upper bound for the left hand side, without making any assumptions about the relative sizes of $H$ and $N$.  We note that the range for $H$ in Theorem~\ref{thm:Av G'}
  is significantly broader than in
Theorem~\ref{thm:Av F'},
which seems counterintuitive, since we have less  averaging in   Theorem~\ref{thm:Av G'}. However, the problem stems from the fact that  the values of polynomials $f(n)$ with 
$f \in \cFH$ and $1\le n \le N$ could be of order  $H N^k$, while for $f \in \cGH$ they are 
of much smaller  order  $H + N^k$, which has a strong effect on the set of moduli for which we have to 
sieve.

\begin{ack}
This work started during a very enjoyable visit by the second author  to  
IST Austria
whose hospitality   and support are very much appreciated.
The first author was
supported by FWF grant 
P~36278 and the second author by ARC grant DP230100534.
\end{ack}

\section{Preliminaries}

\subsection{Notation and conventions}
\label{sec:not} 
	We adopt the Vinogradov notation $\ll$,  that is,
\[C\ll D~\Longleftrightarrow~C=O(D)~\Longleftrightarrow~|C|\le cD\]
for some constant $c>0$ which is allowed to depend on the 
integer parameter $k\ge 1$ and the real parameters $A, \ve >0$.
For a finite set $\cS$ we use $\# \cS$ to denote its cardinality.
We also write $\e(z)=\exp(2\pi i z)$ and $\e_m(z)=\e(z/p)$. 
In what follows we 
make frequent use of the bound
\begin{equation}\label{eq: tau}
\tau(r) \leq |r|^{o(1)},  \qquad \text{for $r \in \Z, \ r \ne 0$,}
\end{equation}
for the divisor function $\tau$, and its cousins, 
as explained in~\cite[Equation~(1.81)]{IwKow}, for example. 
Finally, as usual, $\mu(r)$ denotes the M{\"o}bius function.

\subsection{Lattice points in boxes} 

We 
use some tools from the geometry of numbers, as explained in Cassels~\cite{Cass}.
Let 
\[
\Lambda = \left \{u_1 \vb_1 + \ldots + u_s\vb_s: ~ \(u_1, \ldots, u_s\) \in \Z^s \right\}
\]
be  an $s$-dimensional lattice 
defined by $s$ linearly independent vectors $\vb_1, \ldots, \vb_s\in \Z^s$.
We denote by $\lambda_1 \le  \ldots \le \lambda_s$ the successive minima
of  $\Lambda$, which for $j=1, \ldots, d$ is  defined to be
\[
\lambda_j =\inf\{\lambda >0:~\lambda \fB_s\ \text{contains $j$ linearly independent elements of } \Lambda\},
\]
where $\lambda \fB_s$ is the homothetic image of $\fB_s$ of the unit ball $\fB_s\subseteq \R^s$ 
at the origin with the coefficient $\lambda$.

We also recall that
the discriminant $\Delta$ of $\Lambda$ is an invariant that 
is independent of the choice of basis for $\Lambda$. We have 
\begin{equation}
\label{eq: prod lambda}
\Delta \leq  \lambda_1 \ldots \lambda_s \ll \Delta,
\end{equation}
where the implied constant only depends on $s$.

Next, we need the following consequence of the classical result of Schmidt~\cite[Lemma~2]{Schm} 
on counting lattice points in boxes. 

\begin{lemma}
\label{lem: LattPoints}
 Let  $\lambda_1 \le  \ldots \le \lambda_s$ be the successive minima of a full rank  lattice   $\Lambda\subseteq \Z^s$.  
 Then 
 \[
 \# \(\Lambda \cap [-H,H]^s\) \ll \frac{H^s}{\Delta}   + 
 \(\frac{H}{\lambda_1}\)^{s-1} + 1, 
\]
 where the implied constant depends only on $s$.
\end{lemma}

\begin{proof} 
By~\cite[Lemma~2]{Schm}, we have  the following asymptotic formula
\[
\left| \# \(\Lambda \cap [-H,H]^s\) - \frac{\(2H+1\)^s}{\Delta}\right |  \ll 
\sum_{j=0}^{s-1} \frac{H^{j} }{\lambda_1\ldots  \lambda_j}.
\]
It follows that 
 \[
 \# \(\Lambda \cap [-H,H]^s\) \ll  \frac{H^s}{\Delta}   + 
 \sum_{j=0}^{s-1} \(\frac{H}{\lambda_1}\)^{j}, 
\]  
which yields  the result.
\end{proof} 

\subsection{Univariate polynomial congruences}
\label{sec:UnivarPolyCong}

For  $f \in \Z[X]$, we define 
\begin{equation}\label{eq:def-rho}
\rho_{f} (m) = 
\# \{n \in \ZZ/m\ZZ:~f(n) \equiv 0 \bmod m\}.
\end{equation}
%
%
To estimate $\rho_{f} (m)$ we may use the following result. 

\begin{lemma}\label{lem:stewart}
Let $p$ be a prime and let $k\in \N$. Let $f\in \mathbb{Z}[X]$ be a polynomial of degree $k$. Assume that $\Delta_f\neq 0$ and $f$ has content coprime to $p$. Then 
\[
\rho_f(p^j)\leq k \min\left\{p^{j(1-\frac{1}{k})},  p^{j-1}\right\}.
\]
Additionally, if $p\nmid \Delta_f$, then
$\rho_f(p^k)\leq d$. 
\end{lemma}

\begin{proof}
The final bound is a straightforward consequence of  Lagrange's theorem and Hensel's lemma.
For the remaining bounds,  the first bound  follows from Corollary 2 and Equation (44) of Stewart~\cite{stewart} and  the second bound is a consequence of Lagrange's theorem. 
\end{proof}

The next bound is an easy consequence of Lemma~\ref{lem:stewart} and the Chinese remainder theorem. 

\begin{cor}
\label{lem:PolyCong-Compl}  
Let $f \in \Z[X]$ be a polynomial of degree $k$ with content $1$. For any  square-free positive integer 
 $q$ we have 
\[
 \rho_{f} (q^2)   \le q^{o(1)} \gcd\(\Delta_f, q\).
\]  
\end{cor}

Note that the bound of Lemma~\ref{lem:PolyCong-Compl}   also 
holds for $\Delta_f=0$. We  
also need to look at averages of $\rho_f(m)$, for which we require the following useful result.

\begin{lemma}\label{lem:av-rho}
Let $f \in \Z[X]$ be a polynomial of degree $k$ with $\Delta_f\neq 0$ and content $1$. 
Then 
\[
\sum_{m\leq M} \rho_f(m) \leq M^{1+o(1)},
\]  
uniformly over $f$.
\end{lemma}

\begin{proof}
Any $m\in \N$ admits a factorisation $m=e_1e_2^2\ldots  e_{k-1}^{k-1}h$ where
\[
\mu^2(e_1\ldots  e_{k-1})=\gcd(e_1\ldots e_{k-1},h)=1
\]  
and $h$ is $k$-full.  Combining 
Lemma~\ref{lem:stewart}  with the Chinese remainder theorem, we deduce that 
\[
\rho_f(m)
=\rho_f(e_1)\rho_f(e_2^2)\cdots \rho_f(e_{k-1}^{k-1}) \rho_f(h)
\leq m^{o(1)} \cdot e_2e_3^{2}\cdots e_{k-1}^{k-2} h^{1-1/k}.
\]  
Hence we deduce that
\begin{align*}
\sum_{m\leq M} \rho_f(m) 
&\leq M^{o(1)}
\sum_{h} h^{1-1/k}\sum_{e_{k-1}} e_{k-1}^{k-2}\cdots \sum_{e_2} e_2\cdot \frac{M}{e_2^2\cdots e_{k-1}^{k-1} h}\\
&\leq M^{1+o(1)}
\sum_{h} \frac{1}{h^{1/k}}\sum_{e_{k-1}} \frac{1}{e_{k-1}}\cdots \sum_{e_2} \frac{1}{e_2}.
 \end{align*}
The sums over $e_2,\ldots ,e_{k-1}$ contribute $O((\log M)^{k-2})$. Moreover, the sum over $h$
contributes $O(\log M)$,  since there are $O(B^{1/k})$ $k$-full positive integers in the dyadic interval $(B/2,B]$, for any $B\geq 1$ and 
the result follows. 
 \end{proof}

\subsection{Polynomial values with a large square divisor}
\label{sec:SqDiv}

For a given polynomial $f\in \ZZ[X]$ of degree $k$, we denote by $\Delta_f\in \ZZ$ its discriminant. 
This is a homogeneous polynomial of degree $k(k-1)/2$ in the coefficients of $f$. 
In this section we 
are  interested in the size of 
\[
Q_f(S,N)=\#\left\{ (n,r,s)\in \ZZ^3: 1\leq n\leq N, ~1\leq s\leq S, ~f(n) = s r^2\right\},
\]  
for given $N,S\geq 1$.
For a  polynomial $f\in \cFH$ with $\Delta_f\neq 0$, this quantity has been estimated in~\cite[Theorem~1.3]{LuSh}, with the outcome that 
\[
Q_f(S,N) \leq    N^{1/2 } S^{3/4} (HN)^{o(1)} .
\]  
The following result improves on this via the determinant method.

\begin{lemma}
\label{lem:LargeSquareDiv}  
Let
$f\in \cFH$ such that $f$ is irreducible. 
Then 
\[
Q_f(S,N) \leq 
 \left(N^{1/2}S^{1/2}+S\right) (HN)^{o(1)} .
\]  
\end{lemma}

\begin{proof} 
Note that the bound is trivial if $S>N$ and so we may proceed under the assumption that $S\leq N$.
We fix a choice of $s$ and begin by breaking into residue classes modulo $s$, giving
\[
Q_f(S,N)\leq \sum_{0<s\leq S} 
\sum_{\substack{ \nu=0\\ f(\nu)\equiv 0 \bmod{s}}}^{s-1}   N(s,\nu),
\]  
where
\begin{align*}
N(s,\nu)
&=
\#\left\{ (n,r)\in \ZZ^2:~1\leq n\leq N, ~n\equiv \nu\bmod{s}, ~f(n) = s r^2\right\}\\
&\leq 
\#\left\{ (u,r)\in \ZZ^2:~u \le N/s+1, ~f(\nu+su) = s r^2\right\}.
\end{align*} 
At this point we call upon work of Heath-Brown~\cite[Theorem~15]{cime}.
Given $\ve>0$  and an absolutely irreducible polynomial $F\in \ZZ[u,v]$ of degree $D$,  
this shows that  there are at most
\[
 (UV)^{o(1)}
\exp\left( \frac{\log U\log V}{\log T}\right)
\]  
choices of $(u,v)\in \ZZ^2$ such that $|u|\leq U$, $|v|\leq V$ and $F(u,v)=0$. 
Here $T$ is defined to be the maximum of $U^{e_1}V^{e_2}$, taken over all monomials $u^{e_1}v^{e_2}$ which appear in $F(u,v)$ with non-zero coefficient.  
Moreover,  this is 
uniform over all absolutely irreducible polynomials $F$ of a given degree $D$.
We 
apply this bound with $U=N/s+1$ and 
$F(u,v)= f(\nu+su) - s v^2$, noting that the  absolute irreducibility of $F$ follows from the irreducibility of $f$.
In particular we may take $T\geq V^2$ and it follows that 
\begin{align*}
Q_f(S,N)\leq (NH)^{o(1)} \sum_{0<s\leq S} 
\rho_f(s) \left(\frac{N}{s}+1\right)^{1/2},
\end{align*}
where $\rho_f(s)$ is  defined in~\eqref{eq:def-rho}.
We now appeal to Lemma~\ref{lem:av-rho}, which we  combine with partial summation, to obtain
the desired upper bound 
 on recalling that $S\leq N$.
\end{proof}  

\subsection{Exponential sums and discrepancy}
\label{sec:ExpSumDiscr}

Given  a sequence $\xi_n \in [0,1)$ for $n\in \N$, we 
denote by $\Delta(N)$  its {\it discrepancy}
\[
\Delta(N)=\sup_{ \alpha  \in [0,1)}\left|\#\{n \le N:~\xi_n \le \alpha\} - \alpha N\right|.
\]  
As explained in~\cite[Theorem~2.5]{KuNi}, for example, 
the celebrated {\it  Erd\H{o}s-Tur\'an inequality\/}, 
allows us to give an upper bound on the discrepancy $\Delta(N)$ in terms of exponential sums. 

\begin{lemma}
\label{lem:ET}
Let $\xi_n$, $n\in \N$,  be a sequence in $[0,1)$. Then for any integer  $L\ge 1$,its discrepancy $\Delta(N)$ satisfies
\[
\Delta(N) \ll \frac{N}{L} + \sum_{h =1}^{L}\frac{1}{h} \left| \sum_{n=1}^{N} \e(h\xi_n) \right |. 
\]  
\end{lemma} 

We proceed by recalling some bounds of exponential sums with polynomial arguments. 
We 
make use of  a bound which follows from the recent  spectacular results of  Bourgain, Demeter and Guth~\cite{BDG} (for $k \geqslant 4$) and Wooley~\cite{Wool1,Wool2} (for $k=3$), towards  the  optimal form of the {\it Vinogradov mean value theorem\/}. 

The current state-of-the-art  bounds for {\it Weyl sums\/} has been conveniently summarised by 
Bourgain~\cite{Bourg}. We need the following special case covered by\cite[Theorems~4 and~5]{Bourg},  for which we do not assume anything about the arithmetic structure of the modulus. 

\begin{lemma}
\label{lem:Weyl}
For any fixed polynomial $g\in \Z[X]$ of degree $k \ge 2$ and any integers  $m,N\ge 1$ we have 
\[
\left|\sum_{n=1}^{N} \e_m \(h g(n)\)\right|\le
 N^{1+o(1)} \(\frac{\gcd(h,m)}{m} + \frac{1}{N} + \frac{m}{\gcd(h,m)N^k}\)^{\eta(k)} , 
\]  
where $\eta(k)$ is given by~\eqref{eq:eta-k}. 
\end{lemma}

\section{Solutions to families of polynomial congruences}
\label{sec:PolyCong}

\subsection{Preliminaries}\label{s:prelim}

 Let $\UkH$ be the number of solution to the congruence
\[
a_0 + a_1n+\cdots+a_kn^k \equiv 0 \pmod m,
\]
in the variables
\[
(a_0,\ldots, a_k) \in \BH \mand 1 \le n \le N.
\]
Similarly,  for given $g \in \Z[X]$, let $\WgH$ be the number of solution to the congruence
\[
a + g(n) \equiv 0 \pmod m,
\]
in the variables
\[
a\in \IH \mand 1 \le n \le N.
\]

It is observed in~\cite[Equation~(3.2)]{Shp} that we have the trivial upper bounds 
\[
\UkH \ll H^{k}(H/m+1)N
\]
and 
\begin{equation}
\label{eq:Triv Cong W}
\WgH \ll (H/m+1)N. 
\end{equation}
%
%
Our aim is to improve on these bounds in appropriate ranges of $H$ and $N$. 


\subsection{Using exponential sum bounds  and discrepancy}
Our next result is based on treating the question of estimating 
$\WgH$ as a question of unifomity of distribution and hence we use the 
tools from Section~\ref{sec:ExpSumDiscr}.

\begin{lemma}
\label{lem:Poly Box-2} Let $g \in \Z[T]$ be of degree $k \ge 2$. 
For  any positive integers $H\le m$ and $N$, 
we have 
\[\WgH =  \frac{HN}{m}+ O\(
 N \(\frac{1}{m} + \frac{1}{N} + \frac{m}{N^k}\)^{\eta(k)} (mN)^{o(1)}  \)
 \]
where $\eta(k)$ is given by~\eqref{eq:eta-k}.
 \end{lemma}

\begin{proof} We observe that $W_{g}(m,H,N)$ is the number 
of fractional parts $\{g(n)/m\}$ which fall in the interval $[1-H/m, 1]$.
Hence, on taking $L=N$ in Lemma~\ref{lem:ET}, we obtain 
\begin{equation}
\label{eq:U Delta}
W_{g}(m,H,N) = \frac{HN}{m} + O(\Delta), 
\end{equation}
where
\begin{equation}
\label{eq:Delta}
\Delta \ll  1 +  \sum_{h =1}^{N}\frac{1}{h} \left| \sum_{n=1}^{N} \e_{m}\(\frac{hg(n)}{m}\) \right |.
\end{equation}
Lemma~\ref{lem:Weyl} yields
\begin{align*}
\sum_{h =1}^{N}\frac{1}{h} & \left| \sum_{n=1}^{N} \e_{m}\left(\frac{hg(n)}{m}\right) \right |\\
&  \le  N^{1+o(1)} \sum_{h =1}^{N}\frac{1}{h}
 \(\frac{ \gcd(h,m)}{m} + \frac{1}{N} + \frac{m}{\gcd(h,m)N^k}\)^{\eta(k)} \\
 &\le  N^{1+o(1)} \sum_{h =1}^{N}\frac{1}{h}
 \(\frac{ \gcd(h,m)}{m} + \frac{1}{N} + \frac{m}{N^k}\)^{\eta(k)}  \\
  & \le N^{1+o(1)} \sum_{h =1}^{N}\frac{\gcd(h,m)^{\eta(k)}}{h m^{\eta(k)}} 
 +    N^{1+o(1)} \(\frac{1}{N} + \frac{m}{N^k}\)^{\eta(k)}  \sum_{h =1}^{N}\frac{1}{h} \\
& = \frac{N^{1+o(1)}}{m^{\eta(k)}} \sum_{h =1}^{N}\frac{\gcd(h,m)^{\eta(k)}}{h} 
 +    N^{1+o(1)} \(\frac{1}{N} + \frac{m}{N^k}\)^{\eta(k)}.  \end{align*}
Furthermore, 
\begin{align*}
 \sum_{h =1}^{N}\frac{\gcd(h,m)^{\eta(k)}}{h} 
 \le  
\sum_{r\mid m} r^{\eta(k)}  \sum_{\substack{h =1\\\gcd(h,m) = r}}^{N}\frac{1}{h} 
&\le   
\sum_{r\mid m}  r^{\eta(k)}   \sum_{1 \le h \le N/r}\frac{1}{hr} \\
& \ll \tau(m) \log N  . 
  \end{align*}
Recalling~\eqref{eq: tau},  the lemma follows from~\eqref{eq:U Delta} and~\eqref{eq:Delta}.
\end{proof}

\subsection{Using the geometry of numbers}
We now use Lemma~\ref{lem: LattPoints} to 
estimate $U_{k}(q^2,H,N)$ on average over  square-free integers $q$
in a dyadic interval.

\begin{lemma}
\label{lem: NmbrSol-Aver}  
For $Q \ge 1$, we have  
\[
\sum_{q \sim Q} \mu^2(q)U_{k}(q^2,H,N) \ll
Z(NQ)^{o(1)},
\]
where
\[
Z=\frac{H^{k+1}N}{Q}   + NQ
+
H^kQ+ H^kNQ^{2/(k+1)}.
\]
\end{lemma} 
\begin{proof} 
For  $m,n \in \N$  we define the lattice
\[
\Lambda_{m,n} = \left\{\va \in \Z^{k+1}:~ \va \cdot \vn \equiv 0 \pmod m\right\},
\]
where $\vn= \(1, n, \ldots, n^k\)\in  \N^{k+1}$ and $\va \cdot \vn$ is the scalar product. 
Note that $\Lambda_{m,n} $ has full rank and discriminant 
\begin{equation}
\label{eq: Discr mn}
\Delta_{m,n} = \frac{m}{\gcd(m, 1, n, \ldots, n^k)} = m.
\end{equation}
By Lemma~\ref{lem: LattPoints}, we have 
\[
U_{k}(q^2,H,N) \ll \sum_{1\le n \le N} \(\frac{H^{k+1}}{q^2} + \frac{H^k}{s(q^2,n)^k} +1\)
\]
where $s(q^2,n)$ is the  smallest successive minima of $\Lambda_{q^2,n}$. 
Therefore,
\begin{equation}
\label{eq: U and S}
U_{k}(q^2,H,N) \ll \frac{H^{k+1}N}{q^2}  + N + H^kS_k(q,N),
\end{equation}
where
\[
S_k(q,N) =  \sum_{1\le n \le N} \frac{1}{s(q^2,n)^k}.
\]

Since $s(q^2,n)$ is the  smallest successive minimum of $\Lambda_{q^2,n}$, it follows from~\eqref{eq: prod lambda} and~\eqref{eq: Discr mn} that
$s(q^2,n)^{k+1} \ll  \Delta_{q^2,n} = q^2$, whence
\[
s(q^2,n)  \le q^{2/(k+1)}.
\]
We now define  an integer $I \ll \log Q$ by the inequalities 
\[
2^{I-1} <  Q^{2/(k+1)} \le 2^{I}
\] 
and write  
\begin{equation}
\label{eq: S dyadic} 
S_k(q,N) \le   \sum_{i=0}^I 
 \sum_{\substack{1\le n \le N\\s(q^2,n) \sim 2^i}} \frac{1}{s(q^2,n)^k} \ll 
   \sum_{i=0}^I  2^{-ik} \sum_{\substack{1\le n \le N\\s(q^2,n) \sim 2^i}}  1. 
\end{equation}
Note that if $s(q^2,n) \le t$ then there is a non-zero vector 
$\vc  \in \Lambda_{q^2,n}$ such that $\|\vc\|_2 \le t$.
Therefore
\begin{equation}
\label{eq: S and rho_f} 
\sum_{\substack{1\le n \le N\\s(q^2,n) \sim 2^i}}  1 \leq    \sum_{\substack{\vc = (c_0, \ldots, c_k)\in \Z^{k+1} \\ 0 < \|\vc\|_2 \le  2^{i}}}  \rho_{f_\vc} (q^2, N),
\end{equation}
where 
\[
f_\vc(X) =  c_0 + c_1X+ \cdots + c_k X^k
\]
and 
\begin{equation}\label{eq:def-rhof}
\rho_{f} (m, N) = 
\# \{n \in [1, N]:~f(n) \equiv 0 \bmod m\}
\end{equation}  

Using~\eqref{eq: S dyadic} and then changing the order of summation in~\eqref{eq: S and rho_f}, we obtain
\begin{equation}
\label{eq: S and R}
\sum_{q \sim Q} \mu^2(q)S_k(q,N) \ll
  \sum_{i=0}^I  2^{-ik} R\(Q,N,2^{i}\), 
\end{equation}
  where
\[
R(Q,N,t) = 
\sum_{\substack{\vc\in \Z^{k+1} \\ 0 < \|\vc\|_2 \le  t}}
\sum_{q \sim Q} \mu^2(q)\rho_{f_\vc} (q^2, N). 
\]
But clearly 
\begin{align*}
R(Q,N,t)
&\leq \sum_{\substack{\vc \in \Z^{k+1} \\ 0<\|\vc\|_2 \le t}}
\sum_{\substack{q \le  Q}}  \sum_{\substack{n\leq N\\ q^2\mid f_\vc(n)}}1\\
&\leq \sum_{n\leq N} 
 \sum_{\substack{\vc \in \Z^{k+1} \\ \|\vc\|_2 \le t}} \#\left\{q\leq Q:~q^2\mid f_{\vc}(n)\right\}.
 \end{align*}
If $f_\vc(n)=0$ then there are $Q$ choices for $q$. Let $\vc$ be a non-zero vector such that $f_\vc$ has a root over $\mathbb{Z}$. Then $f_\vc$ must be reducible over $\mathbb{Z}$.
There are at most
$t^{k+o(1)}$ choices of non-zero vectors $\vc$ for which $f_\vc$
is reducible over $\mathbb{Z}$, 
on appealing to work of Kuba~\cite{kuba}. (See also~\cite{Dub} and the references therein.) 
Moreover,  each such vector $\vc$ yields at most $k$ choices for $n$.
If $f_\vc(n)\neq 0$, on the other hand,  then we have at most $(tN)^{o(1)}$ choices for $q$ by the divisor bound~\eqref{eq: tau}.
  Hence we arrive at the bound
\[
  R(Q,N,t)\leq \(Qt^{k}+
Nt^{k+1} \)(Nt)^{o(1)}.
\]

On  returning to~\eqref{eq: S and R}, 
we therefore obtain 
\begin{align*}
\sum_{q \sim Q} \mu^2(q)S_k(q,N)
& \le 
\sum_{i=0}^I   \(Q + 
N 2^{i}
\) (2^iQ)^{o(1)}\\
&\le 
\left(
Q+
NQ^{2/(k+1)}
\right)
Q^{o(1)}.
\end{align*}
We substitute this into~\eqref{eq: U and S} and sum over $q$. This yields
\begin{align*}
\sum_{q \sim Q}\mu^2(q) &
U_{k}(q^2,H,N) \leq
\frac{H^{k+1}N}{Q}  + NQ 
+
\left(
Q+NQ^{2/(k+1)} \right)H^k
(NQ)^{o(1)}, 
\end{align*}
and the result now follows.
\end{proof}

\section{Proofs of main results} 

\subsection{Proof of Theorem~\ref{thm:Av F'}}

Fix a choice of $A\geq 1$. We 
proceed under the assumption that $N^{1/k} \leq H\leq N^A$, and allow all of our implied constants to depend on $A$. 
There are $O(H^{k+o(1)})$ choices of polynomials $f \in \cFH$ for which $f$ fails to be irreducible, 
the latter bound following from work of Kuba~\cite{kuba}, as in the proof of Lemma~\ref{lem: NmbrSol-Aver}.
The overall contribution from such $f$ is therefore $O(H^{k+o(1)}N)$.
We may henceforth  restrict to the set $\cFHs$ of irreducible polynomials 
$f \in \cFH$. 

Our argument proceeds along standard lines, beginning with 
an application of M\"obius inversion to interpret 
\[
\mu^2(n)=\sum_{d^2\mid n} \mu(d)
\]
as a sum over divisors. This leads to the expression
\[
S_f(H) =\sum_{d \ll \sqrt{H N^k}} \mu(d) \rho_f(d^2,N).
\]
where   $\rho_f(m,N)$
 is given by~\eqref{eq:def-rhof}.
Hence, for  arbitrary  $E\ge D\ge 1$, we have
\[
S_f(H) =M_f(H) + O(R_f^{(1)}(H))+O(R_f^{(2)}(H)), 
\]
where 
\[
M_f(H) =\sum_{d \le D} \mu(d) \rho_f(d^2,N)
\]
and 
\[
R_f^{(1)}(H) =\sum_{D <d \leq E}  \mu^2(d) \rho_f(d^2,N), \quad 
R_f^{(2)}(H) =\sum_{E <d \ll \sqrt{H N^k}}  \mu^2(d) \rho_f(d^2,N).
\]
To begin with, 
it is shown in~\cite[Equation~(4.8)]{Shp} that 
\begin{equation}
\label{eq:Mfh}
 M_f(H) = c_f N + O\(DH^{o(1)} + ND^{-1}H^{o(1)}\), 
\end{equation}
where the implied constant in this estimate depends only on $k$.

Next, 
on recalling the notation 
 $Q_f(S,N)$ that has been defined  in Section~\ref{sec:SqDiv},
we may write 
\[
R_f^{(2)}(H) \le
Q_f(S,N)+Q_{-f}(S,N),
\]
for some $S \ll HN^k/E^2$. 
Thus, Lemma~\ref{lem:LargeSquareDiv}  implies that 
\begin{equation}
\label{eq: U2 bound}
R_f^{(2)}(H)   \le   \(
 N^{1/2} \(\frac{HN^k}{E^2}\)^{1/2}+
\frac{ HN^k}{E^2}
  \) (HN)^{o(1)}.
\end{equation}

Turning to the remaining error term $R_f^{(1)}(H)$, we are only able to estimate it well on average over $f \in \cFHs$.
Recall the definition of $U_k(d^2,H,N)$ from
Section~\ref{s:prelim}.
 On changing the order of summation, we obtain 
\[
\sum_{f \in \cFHs}  R_f^{(1)}(H) \leq  \sum_{D <d \leq E}  \mu^2(d) U_k(d^2,H,N).
\]
To estimate the right hand side, we   use Lemma~\ref{lem: NmbrSol-Aver}.
After splitting the summation range in dyadic intervals, we  derive 
\[
\sum_{f \in \cFHs}  R_f^{(1)}(H) 
\le \( \frac{H^{k+1}N}{D}   + NE +  H^kE +  
H^kN E^{2/(k+1)} \)  H^{o(1)}.
\]
Since we are assuming 
$N\leq H^k $, we may drop the second term in this estimate. 
Hence
\begin{equation}
\label{eq: U1 bound}
\sum_{f \in \cFHs}  R_f^{(1)}(H) 
\le \( \frac{H^{k+1}N}{D}   +H^kE +  
H^kN E^{2/(k+1)} \)  H^{o(1)}.
\end{equation}

On accounting for the $O(H^{k+o(1)})$ choices of $f \in \cFH\setminus \cFHs$, it therefore follows 
from~\eqref{eq:Mfh}--\eqref{eq: U1 bound}
that 	
\begin{align*}
\sum_{f \in \cFH}  \left| S_f(N) - c_f N\right| \leq~&  H^{k+o(1)}N + \left(D+ND^{-1}\right)H^{k+1+o(1)}  
 \\
 &\quad +
\(  N^{1/2} \(\frac{HN^k}{E^2}\)^{1/2}+
\frac{ HN^k}{E^2}\) H^{k+1+o(1)}\\
& \qquad +  
\(
 \frac{H^{k+1}N}{D}   +H^kE +  
H^kN E^{2/(k+1)} \)  H^{o(1)}.
\end{align*} 
Hence, on noting that $\#\cFH\gg H^{k+1}$, it follows that 
\[
\frac{1}{\#\cFH}
\sum_{f \in \cFH}  \left| S_f(N) - c_f N\right| \le \Delta H^{o(1)},
\]
where 
\begin{align*}
\Delta = D+
\frac{N}{D}   +\frac{E}{H} + \frac{N E^{2/(k+1)}}{H}
+ 
\frac{H^{1/2}N^{(k+1)/2}}{E}+
\frac{ HN^k}{E^2}.
\end{align*}  
We take $D=N^{1/2}$, leading to 
\[
\frac{1}{\#\cFH}
\sum_{f \in \cFH}  \left| S_f(N) - c_f N\right| \le \Delta_0 H^{o(1)},
\]
where
\[
\Delta_0=\inf_{\sqrt{N}\leq E\ll \sqrt{HN^k}}\Delta_0(E)
\]
and 
\[
\Delta_0(E)=N^{1/2}+
\frac{E}{H} + \frac{N E^{2/(k+1)}}{H}
+\frac{H^{1/2}N^{(k+1)/2}}{E}+
\frac{ HN^k}{E^2}.
\]
We expect the dominant contribution to come from the second and fourth terms and so we choose 
\[
E=
\min\left\{
H^{3/4}N^{(k+1)/4},
\sqrt{HN^k}\right\},
\] in order to minimise their contribution. Note that $E\geq \sqrt{N}$ with this choice.
This therefore leads to the bound 
\begin{align*}
\Delta_0
&\ll 
N^{1/2}+
\frac{N^{(k+1)/4}}{H^{1/4}}+
 \frac{N^{3/2} }{H^{1-3/(2k+2)}}
+
\frac{ N^{(k-1)/2}}{H^{1/2}},
\end{align*}
which thereby concludes the proof of the following result.

\begin{theorem}
\label{thm:Av F} 
Let $A\geq 1$ and $k\ge 2$ 
be fixed and assume that 
$H,N\to \infty$ in such a way that  
$N^{1/k}\leq H\leq N^A$.
Then we have
\begin{align*}
\frac{1}{\#\cFH}\sum_{f \in \cFH} & \left| S_f(N) - c_f N\right|\\
&  \le  
\(
N^{1/2}+
\frac{N^{(k+1)/4}}{H^{1/4}}+
 \frac{N^{3/2} }{H^{1-3/(2k+2)}}
+
\frac{ N^{(k-1)/2}}{H^{1/2}}
 \)H^{o(1)}.
\end{align*}
\end{theorem}

To deduce Theorem~\ref{thm:Av F'} we assume that $k\geq 4$ and proceed to assess when each of the terms is $O(N^{1-\delta})$ for some $\delta>0$. The first term is obviously satisfactory. One  sees that the second term and fourth terms are satisfactory if 
$H\geq N^{k-3+\ve}$ for some $\ve>0$. Finally, the third term is only satisfactory if 
$H\geq N^{1/2+3/(4k-2)+\ve}$, but this is implied by the latter condition. 
This completes the proof of Theorem~\ref{thm:Av F'}.

\subsection{Proof of Theorem~\ref{thm:Av G}}

The aim of this section is to estimate the quantity
\[
\Sigma=\sum_{f \in \cGH}  \left| S_f(N) - c_f N\right|.
\]
It is convenient to define 
\[
M = \max\{H,N^k\},
\]
so that $f(n)=a+g(n) = O(M)$ if $f\in \mathcal{G}_g(H)$ and 
 $1 \le n \le N$. 
Mimicking the previous argument and using~\eqref{eq:Mfh},
we obtain 
\begin{equation}
\label{eq: S and W}
\Sigma
 \leq DH^{1+o(1)}  +
\frac{NH^{1+o(1)}}{D} 
 +  \sum_{D <d \le c \sqrt{M}}   W_g(d^2,H,N),
\end{equation} 
where $c>0$ is a 
constant depending only on the polynomial $g$ and 
$W_g(d^2,H,N)$ is defined in Section 
\ref{s:prelim}.

Suppose first that  $H \ge N^k$. Then  we simply apply~\eqref{eq:Triv Cong W}  and get 
\begin{align*}
\Sigma& \ll DH^{1+o(1)}  +\frac{
NH^{1+o(1)}}{D} 
+  \sum_{D <d \le c \sqrt{H}}    (H/d^2+1)N \\
& \le \left(DH+\frac{HN}{D} + H^{1/2} N\right)H^{o(1)}.
\end{align*} 
Taking  $D = N^{1/2}$, we derive
\begin{equation}
\label{eq: large H}
\Sigma \leq
\(HN^{1/2} + H^{1/2} N \)H^{o(1)}
  \le H^{1+o(1)}N^{1/2}, 
\end{equation}
if 
 $H \ge N^k$.

We may henceforth assume that $H \le N^k$ and thus $M = N^k$.
We now choose $D=N^{1/2}$ and two more  parameters $F$ and $E$ with $F \ge E \ge N^{1/2}$. Then we may  write  
\begin{equation}
\label{eq: W1 W2 W3}
 \sum_{N^{1/2} <d \le c \sqrt{M}}   W_g(d^2,H,N) = \fW_1 + \fW_2 +  \fW_3\end{equation} 
 where 
\begin{align*}
\fW_1 &=  \sum_{N^{1/2} <d \le E}   W_k(d^2,H,N), \\
 \fW_2 &=  \sum_{E<d \le F}  W_k(d^2,H,N), \\
 \fW_3 &=  \sum_{F <d \le  c N^{k/2} }  W_k(d^2,H,N) . 
\end{align*}

To begin with,  we appeal to~\eqref{eq:Triv Cong W} to estimate   
\begin{equation}
\label{eq: W1 bound prelim}
\fW_1 \ll   \sum_{N^{1/2} <d \le E}   (H/d^2+1)N 
\ll HN^{1/2} +  EN.
\end{equation} 
It is convenient to choose 
$
E=\max\{H^{1/2}, N^{1/2}\},
$
so that
\begin{equation}
\label{eq: W1 bound}
\fW_1    \ll HN^{1/2} +  H^{1/2} N.
\end{equation} 
Indeed, for $H\le  N$ we have $E=N^{1/2}$ and thus $\fW_1 = 0$, while 
for $H > N$ we have $E = H^{1/2}$ and~\eqref{eq: W1 bound} 
follows from~\eqref{eq: W1 bound prelim}. 

Therefore, 
combining~\eqref{eq: S and W}, \eqref{eq: W1 W2 W3}
and~\eqref{eq: W1 bound}, we obtain
 \begin{equation}
\label{eq: S and W2 W3}
\Sigma\ll H^{1+o(1)} N^{1/2}  + H^{1/2} N 
+  \fW_2 +  \fW_3.
\end{equation} 
It remains to estimate  $\fW_2$ and $\fW_3$. 

To estimate $\fW_2$ we appeal to   Lemma~\ref{lem:Poly Box-2} to derive 
\[
\fW_2  \ll   \sum_{E<  d  \leq F} \( \frac{HN}{d^2}+
 N^{1+o(1)} \(\frac{1}{d^2} + \frac{1}{N} + \frac{d^2}{N^k}\)^{\eta(k)} \), 
\]
where $\eta(k)$ is given by~\eqref{eq:eta-k}.  
Therefore, noticing  that 
we have 
\[
\frac{1}{d^2} <  \frac{1}{N}
\]
for  $d > E \ge N^{1/2}$,
we obtain
 \begin{equation}
\begin{split}
\label{eq: W2 bound}
\fW_2 
& \ll HN/E  + FN^{1 -\eta(k) +o(1)} + F^{1+2\eta(k)} N^{1-k\eta(k)+o(1)}\\
& \leq H N^{1/2}+ FN^{1 -\eta(k) +o(1)} + F^{1+2\eta(k)}  N^{1-k\eta(k)+o(1)}.
\end{split}
\end{equation}

Finally,  as in the proof of Theorem~\ref{thm:Av F}, we treat $\fW_3$ via 
 Lemma~\ref{lem:LargeSquareDiv}. Recalling our assumption $H \le N^k$
 and observing that there are $O(1)$ choices of $f\in \cGH$ that fail to be irreducible, 
 we derive
\begin{equation}
\begin{split}
\label{eq: W3 bound}
\fW_3 
&  \le   \(N + HN^{1/2} \(N^k /F^2\)^{1/2}+HN^k /F^2\) N^{o(1)}\\
&  \le   \(N + HN^{(k+1)/2}/ F+HN^k /F^2\) N^{o(1)}.
\end{split}
\end{equation} 
We now observe that if $N^{(k+1)/2}/F\ge N$, which is equivalent to  $F\leq   N^{(k-1)/2}$,
then  the  bound becomes trivial. Thus we 
always assume that 
\[
F \ge N^{(k-1)/2 }, 
\]
in which case we see that the third term in~\eqref{eq: W3 bound} is dominated by the second term. 

Substituting the bounds~\eqref{eq: W2 bound} and~\eqref{eq: W3 bound}
in~\eqref{eq: S and W2 W3}, we are led to the upper bound 
\begin{align*}
\Sigma
 \leq~& \bigl(H N^{1/2} + H^{1/2} N +  FN^{1 -\eta(k)} 
+ F^{1+2\eta(k)} N^{1-k\eta(k)}+ H N^{(k+1)/2} /F\bigr) H^{o(1)}.
\end{align*}
Since  $F \ge N^{(k-1)/2 }$, we see that  $FN^{1 -\eta(k)}  \le  F^{1+2\eta(k)} N^{1-k\eta(k)}$ and so
\begin{equation}
\begin{split}
\label{eq: fin bound}
\Sigma   \leq~& \bigl(H N^{1/2} + H^{1/2} N + F^{1+2\eta(k)} N^{1-k\eta(k)}  
+ H N^{(k+1)/2} /F\bigr) H^{o(1)}.
\end{split}
\end{equation} 
To optimise~\eqref{eq: fin bound}, we choose
\[
F = \max\left\{\(H N^{(k-1)/2+k\eta(k)}\)^{1/(2+2\eta(k))}, N^{(k-1)/2 }\right\}
\]
for which 
\begin{align*}
F^{1+2\eta(k)} N^{1-k\eta(k)} 
 & =H N^{(k+1)/2} /F\\
 &= H^{1-\frac{1}{2+2\eta(k)}}
 N^{(k+3)/4-\frac{(k+1)\eta(k)}{4+4\eta(k)}}.
\end{align*} 
After substitution in~\eqref{eq: fin bound}, 
this completes our treatment of the case 
$H \le N^k$. Taken together with~\eqref{eq: large H}, 
and observing that 
$\#\cGH\gg H$, 
this therefore concludes the proof of the following theorem.

\begin{theorem}
\label{thm:Av G} 
For a fixed polynomial $g \in \Z[X]$ of degree $k\ge 2$, we have
\begin{align*}
\frac{1}{\#\cGH}&\sum_{f \in \cGH} \left| S_f(N) - c_f N\right| \\
&\qquad \quad  \le  
\(N^{1/2} + \frac{N}{H^{1/2}} 
+ \frac{N^{(k+1)/2-\eta(k)}}{H}+\frac{N^{(k+3)/4-\frac{(k+1)\eta(k)}{4+4\eta(k)}}}
{H^{\frac{1}{2+2\eta(k)}}}
\) H^{o(1)},
\end{align*}
where $\eta(k)$ is given by~\eqref{eq:eta-k}.  
\end{theorem}

Finally, to deduce Theorem~\ref{thm:Av G'} we need to discover when each of the terms is $O(N^{1-\delta})$ for some $\delta>0$. The first term is obviously satisfactory. One  sees that the second term is satisfactory if 
$H\geq N^{\ve}$ for any $\ve>0$. Finally, 
the fourth term is only satisfactory if 
$H\geq N^{(k-1)/2+\eta(k)+\ve}$, 
under which assumption the third term is also satisfactory. 
This  therefore completes the proof of Theorem~\ref{thm:Av G'}.

\end{document}